\documentclass[a4paper]{amsart}
\usepackage[all]{xy}
\usepackage[colorlinks=true]{hyperref}
\usepackage{enumerate}
\usepackage{amssymb}
\usepackage{amssymb}
\usepackage{amsfonts}

\newtheorem{thm}{Theorem}[section]

\newtheorem{prop}[thm]{Proposition}
\newtheorem{lem}[thm]{Lemma}
\newtheorem{cor}[thm]{Corollary}

\theoremstyle{definition}
\newtheorem{defn}[thm]{Definition}
\newtheorem{ex}[thm]{Example}

\theoremstyle{remark}

\newcommand{\Tt}{\mathbb T}
\renewcommand{\d}{\mathrm d}
\newcommand{\R}{\mathcal{R}}

\newcommand{\set}[1]{\left\{#1\right\}}
\newcommand{\eval}[1]{\left\langle#1\right\rangle}
\newcommand{\brr}[1]{\left[#1\right]}

\newcommand{\X}{\ensuremath{\mathfrak{X}}}
\newcommand{\F}{\ensuremath{\mathcal{F}}}

\newcommand{\ds}{\displaystyle}

\renewcommand{\d}{\mathrm d}               
\newcommand{\Lie}{\boldsymbol{\pounds}}    

\newcommand{\smalcirc}{\mbox{\,\tiny{$\circ $}\,}}  

\DeclareMathOperator{\red}{red}         
\DeclareMathOperator{\graf}{graph}			
\DeclareMathOperator{\ch}{char}         
\DeclareMathOperator{\modular}{mod}     
\renewcommand{\mod}{\modular}
\renewcommand{\top}{\text{\rm top}}

\newcommand{\al}{\alpha}
\newcommand{\be}{\beta}

\renewcommand{\gg}{\mathfrak{g}}        

\newcommand{\B}{\mathcal{B}}

\DeclareMathOperator{\Rep}{Rep}         

\usepackage[dvips]{graphicx}

\begin{document}

\title{The modular class of a Dirac map}

\author{Raquel Caseiro}
\address{CMUC, Department of Mathematics, University of Coimbra
}
\email{raquel@mat.uc.pt}

\thanks{This work was partially supported by the Centre for Mathematics of the
University of Coimbra -- UID/MAT/00324/2013, funded by the Portuguese
Government through FCT/MEC and co-funded by the European Regional Development Fund through the Partnership Agreement PT2020.}


\begin{abstract}
In this paper we study the modular classes of Dirac manifolds and of Dirac maps, and we discuss their basic properties. We apply these results
to explain the relationship between the modular classes of the various structures involved in the reduction of a Poisson manifold under the action
by of a Poisson Lie group.
\end{abstract}

\maketitle



\section{Introduction}

The modular class of a Poisson manifold \cite{Weinstein1} is a class lying in the first Poisson cohomology group, that obstructs the existence of a volume form invariant under all Hamiltonian vector fields. More generally, the modular class of a Lie algebroid $A$ can be introduced by considering the fiberwise linear Poisson structure on the dual bundle $A^*$. The resulting class, denoted $\mod A$, lives in the first Lie algebroid cohomology of $A$ with trivial coefficients.  In \cite{ELW}, the authors introduced a certain representation of the Lie algebroid $A$ and showed that the modular class $\mod A$ could be seen as a characteristic class of this representation. This approach shed some new light into this concept and led to many other developments. For example, using this approach, in \cite{GMM, KLW} the authors defined the modular class of a Lie algebroid morphism $\Phi:A\to B$ and showed that it is the obstruction to the existence of a canonical relation, given by $\Phi$, between the modular classes of $A$ and of $B$. More recently, the modular class of a Poisson map was introduced in \cite{CLF} and then extended to any Lie algebroid comorphism in \cite{RC} and to skew-algebroid relations in \cite{Grab}.

In the present paper we study the modular classes of Dirac manifolds and Dirac maps. A Dirac manifold  $(M,L)$ is a manifold $M$ equipped with an integrable isotropic subbundle $L\subset TM\oplus T^*M$ of the generalized tangent bundle. Geometrically, they correspond to (possibly singular) presymplectic foliations of $M$. The Dirac structure $L$ has the structure of a Lie algebroid so it has automatically a modular class. A Dirac structure $L$ is said to be reducible if its characteristic distribution $L\cap TM$ is regular and the associate foliation $\F$ is simple, so that the leaf space $M/\F$ is smooth. A reducible Dirac structure $L$ induces a Poisson structure on $M/\F$ and it is completely determined by it: in fact, we will see that $L$ is the pullback Lie algebroid of $T^*(M/\F)$ and this allows us to conclude that the modular class of $L$ is the pullback of the modular class of the Poisson manifold $M/\F$.

Next, we turn to morphism between Dirac manifolds are known as  Dirac maps. There are two types of morphisms: ``forward'' and ``backward'' Dirac maps. These are in fact canonical relations, in the sense of Weinstein, so inspired by the work of \cite{Grab} we define their modular classes.  We show that they are the obstruction to the existence of a canonical relation between the modular classes of the Dirac manifolds involved. Our definition somehow generalizes the modular classes of both Lie algebroid morphisms and comorphisms.

We end the paper by applying our results on the modular class of Dirac manifolds and Dirac maps to study the behaviour of the modular class of a Poisson manifold under reduction by an hamiltonian action of a Poisson Lie group. This case was already considered in \cite{CLF}, but we will see that our approach using Dirac geometry gives a better understanding of the relationship between the modular classes involved.


\section{Modular classes on Lie algebroids}

\label{subsec:mod:Lie:algbrd}                                  %


In this section we begin by recalling some facts about Lie algebroids  and their modular classes (see, e.g., \cite{Mackenzie}, \cite{YKS}).

\subsection{Modular class of Lie algebroids}
Let $A\to M$ be a Lie algebroid over $M$, with anchor $a:A\to TM$
and Lie bracket $[\cdot,\cdot]:\Gamma(A)\times\Gamma(A)\to\Gamma(A)$.
We will denote by
$\Omega^k(A)\equiv\Gamma(\wedge^k A^*)$ the $A$-forms and by $\X^k(A)\equiv\Gamma(\wedge^k A)$
the $A$-multivector fields.
Given a section $\Pi\in \X^2(A)$, we denote by $\Pi^\sharp$ the bundle map $A^*\to A$ defined by
\[
\eval{\Pi^\sharp (\al),\be}=\Pi(\al,\be).
\]

The Lie algebroid structure on $A$, defines a differential $\d_A$ on $\Omega^\bullet(A)$ and  the
cohomology of the complex $(\Omega^\bullet(A),\d_A)$ is called Lie algebroid cohomology of $A$ (with trivial coefficients)
and will be denoted by $H^\bullet(A)$.


A representation of $A$ is a vector bundle $E\to M$ together with a flat $A$-connection $\nabla$ (see, e.g,
\cite{Fernandes1}).
The usual operations $\oplus$ and $\otimes$ on vector bundles turn the space of
representations $\Rep(A)$ into a semiring.
Given a morphism of Lie algebroids:
\[
\xymatrix{ A\ar[r]^{\Phi}\ar[d]& B\ar[d]\\
M\ar[r]_{\phi}& N}
\]
there is a pullback operation on representations $E\mapsto \phi^{!}E$, which gives a morphism of
rings $\phi^{!}:\Rep(B)\to\Rep(A)$.

Representations have characteristic classes (see, e.g., \cite{Crainic}). Here we are
interested in line bundles $L\in\Rep(A)$ for which the only characteristic class can be obtained as follows:
Assume first that $L$ is orientable, so that it carries a nowhere vanishing section $\mu\in\Gamma(L)$. Then:
\[ \nabla_X \mu=\langle \al_\mu,X\rangle\mu,\quad  X\in\X(A).\]
The 1-form $\al_\mu\in\Omega^1(A)$ is $\d_A$-closed and  is called the {characteristic cocycle} of the
representation $L$. Its cohomology class is independent of the choice of section $\mu$ and
defines the characteristic class of the representation $L$:
\[ \ch(L):=[\al_\mu]\in H^1(A).\]
One checks easily that if $L,L_1,L_2\in\Rep(A)$, then:
\[ \ch(L^*)=-\ch(L),\qquad \ch(L_1\otimes L_2)=\ch(L_1)+\ch(L_2).\]
Also, if $(\Phi,\pi):A\to B$ is a morphism of Lie algebroids, and $L\in\Rep(B)$ then:
\[ \ch(\phi^{!}L)=\Phi^*\ch(L),\]
where $\Phi^*:H^\bullet(B)\to H^\bullet(A)$ is the map induced by $\Phi$ at the level of cohomology.
If $L$ is not orientable, then one defines its characteristic class to be the one half that
of the representation $L\otimes L$, so the formulas above still hold, for non-orientable line bundles.

Every Lie algebroid $A\to M$ has a canonical representation on the line bundle $L_A=\wedge^{\top}A\otimes\wedge^{\top}T^*M$:
\[ \nabla_X (\omega\otimes\mu)=\Lie_X\omega\otimes\mu+\omega\otimes\Lie_{\rho(X)}\mu. \]
Then one sets \cite{Weinstein1,ELW}:
\begin{defn}
The \textbf{modular cocycle} of a Lie algebroid $A$  relative to a nowhere vanishing section
$\omega\otimes\mu\in\Gamma(\wedge^{\top}A\otimes\wedge^{\top}T^*M)$ is the characteristic cocycle
$\al_{\omega\otimes\mu}$ of the representation $L_A$. The \textbf{modular class} of $A$ is the characteristic class:
\[ \mod(A):=[\al_{\omega\otimes\mu}]\in H^1(A).\]
\end{defn}

\begin{ex}
\label{ex:mod:Poisson}
For any Poisson manifold $(M,\pi)$ there is a natural Lie algebroid
structure on its cotangent bundle $T^*M$: the anchor is
$a=\pi^\sharp$ and the Lie bracket on sections of $A=T^*M$, i.e.,
on $1$-forms, is given by:
\[ [\al,\be]_\pi=\Lie_{\pi^\sharp\al}\be-\Lie_{\pi^\sharp\be}\al-d\pi(\al,\be).\]
The Poisson cohomology of $(M,\pi)$ is just the Lie algebroid cohomology
of $T^*M$, and the modular class of the Lie algebroid $T^*M$ is twice the modular class of the Poisson manifold $M$ (see \cite{Weinstein1} for the definition of modular class of a Poisson manifold):
\[ \mod(T^*M)=2\mod(M,\pi).\]
\end{ex}

\subsection{Modular class of morphisms}
If $\Phi:A\to B$ is a morphism of Lie algebroids covering a map $\phi:M\to N$,
the induced morphism at the level of cohomology $\Phi^*:H^\bullet(B)\to H^\bullet(A)$,
in general, does not map the modular classes to each other. Therefore one sets (\cite{GMM, KLW}):

\begin{defn}
The \textbf{modular class} of a Lie algebroid morphism $\Phi:A\to B$ is the cohomology
class defined by:
\[ \mod(\Phi):=\mod(A)-\Phi^*\mod(B)\in H^1(A).\]
\end{defn}


The modular class of a Lie algebroid
morphism $(\Phi,\phi):A\to B$ may also be seen as the characteristic class of a representation (\cite{KLW}):
\[ \mod(\Phi)=\ch(L_A\otimes\phi^{!}(L_B)^*).\]
It is easy to see that Lie algebroid isomorphisms have vanishing modular classes.

Let $A\to N$ be a Lie algebroid over $N$ with anchor $a:A\to TN$ and let $\phi:M\to N$ be a smooth map such that
$$\phi^{!!}A=\set{(X,v)\in T_x M\times A_{\phi(x)}: a(v)=\phi_*X,\, x\in M}$$
is a vector subbundle of $\ds TM\oplus\phi^! A$ (i.e. it has constant rank), where $\phi^!A\to M$ denotes the pullback vector bundle.

The vector bundle  $\ds \phi^{!!}A$ carries a structure of Lie algebroid over $M$, called the \textbf{pullback Lie algebroid of $A$},
given by:
\begin{enumerate}
\item the anchor $a^!:\phi^{!!}A\to TM $ is the  the projection onto the first factor;
\item the Lie bracket on sections is given by:
\begin{align*}
\brr{(X,f\otimes \al)  , (Y,g\otimes \be) }=\brr{X,Y} + fg\otimes \brr{\al,\be}_A + X(g)\,\be- Y(f)\, \al ,
\end{align*}
\end{enumerate}
for $f,g\in C^\infty(M), \al,\be\in\Gamma(A), X,Y\in \X(M).$

This Lie algebroid structure turns the natural bundle map $(\Phi,\phi):\phi^{!!}A\to A$ into a Lie algebroid morphism, so we have a diagram of morphisms:
\begin{equation*}
\xymatrix{
\phi^{!!}A\ar[r]^{\Phi}\ar[dr]_{a^!}& A\\
& TM}
\end{equation*}

\begin{prop}(\cite{KLW})\label{prop:KLW:pullback}
Let $\phi:M\to N$ be a submersion and $A\to N$ a Lie algebroid. Then $\phi^{!!} A$ is a Lie algebroid over $M$ and the modular class of the Lie algebroid morphism $(\Phi,\phi):\phi^{!!}A\to A $ vanishes.
\end{prop}

This means that $\ds \mod(\phi^{!!}A)=\Phi^*(\mod(A))$. In particular, one can always choose modular cocycles on $\phi^{!!}A$ and on $A$ that are related by $\Phi^*$.

\subsection{Modular class of comorphisms}

Let $A\to M$ and $B\to N$ be two Lie algebroids. A {comorphism} between $A$ and $B$ covering $\phi:M\to N$ is a vector bundle map $\Phi:\phi^! B\to A$, from the pullback vector bundle $\phi^!B$ to $A$, such that the following two conditions hold:
 \[
 \brr{\bar\Phi X,\bar \Phi Y}=\bar\Phi\brr{X,Y},
 \]
 and
 \[ \phi_*\smalcirc a_A(\bar\Phi X)=a_B (X),\]
for $X,Y\in \Gamma(B)$, where $\bar\Phi:\Gamma(B)\to \Gamma(A)$ is the natural  map on sections induced by $\Phi$.

Equivalently, we may say that $(\Phi,\phi)$ is a Lie algebroid comorphism if and only if $\Phi^*:A^*\to B^*$ is a Poisson map for the natural linear Poisson  structures on the duals of the Lie algebroids.

When one has a Lie algebroid comorphism
 $\Phi:\phi^!B\to A$, the pullback vector bundle $\phi^!B\to M$ carries a natural  Lie algebroid structure  characterized by:
 \[
 \brr{X^!,Y^!}=\brr{X,Y}^!
 \]
 and
 \[ a(X^!)=a_A(\bar\Phi X^!),\]
 for $X,Y\in \Gamma(B)$, $X^!=X\smalcirc \phi  \in\Gamma(\phi^!B)$ and $Y^!=Y\smalcirc \phi  \in\Gamma(\phi^!B)$.

 For this structure, the natural maps
\begin{equation*}
\xymatrix{
\phi^{!}B\ar[r]^{\Phi}\ar[dr]_{j}& A
\\
&B}
\end{equation*}
are Lie algebroid morphisms.

\begin{defn}(\cite{CLF, RC})
 The \textbf{modular class} of the Lie algebroid comorphism $\Phi:\phi^!B\to A$  is the cohomology class:
\[ \mod(\Phi):=\Phi^*\mod(A)-j^*\mod(B)\in H^1(\phi^{!}B). \]
\end{defn}

Clearly,
 $\mod(\Phi)$ gives the obstruction to the existence of modular cocycles $\al\in\Omega^1(A)$ and $\be\in\Omega^1(B)$, such that
$$
\Phi^*\al=\be\smalcirc \phi.
$$

Looking at  characteristic classes, one notices the natural representation of $\phi^!B$ on the
line bundle $L^\phi:=L_A\otimes \phi^!L_B^*$, and we have:
\[ \mod(\Phi)=\ch(L^\phi).\]

\begin{ex}
A Poisson map $\phi:M\to N$  defines a comorphism between  cotangent bundles: $\ds \Phi:\phi^!T^*N\to T^*M$ such that
$\Phi (\al^!)=\phi^*\al$, where $\al^!=\al\smalcirc\phi\in\Gamma\left(\phi^!(T^*N)\right)$, for all $\al\in\Omega^1(N)$.
The modular class of the Poisson map $\phi$ was defined in \cite{CLF} and we see that it is one half the modular class of the comorphism $\Phi$ induced by $\phi$.
\end{ex}

\section{Dirac manifolds and their modular classes}     %

Dirac structures on manifolds were introduced in \cite{Courant} to unify presymplectic and Poisson structures. In this section we recall some important properties about Dirac manifolds (see, e.g., \cite{Courant, Bursztyn}) and we see how their modular classes behave under Dirac maps.

\subsection{Dirac structures}

Let $M$ be a manifold and let us consider the double (or generalized tangent bundle) $\Tt M=TM\oplus T^*M$ equipped with  natural projections
$$
\rho:\Tt M \to TM\quad \mbox{ and }\quad  \rho_*:\Tt M\to T^*M.
$$
The symmetric fiberwise bilinear  form
\begin{equation}\label{eq:symmetric:bilinear:form}
\eval{X+\al, Y+\be}_+=\frac{1}{2}(\eval{\be,X}+\eval{\al,Y}),
\end{equation}
the   bracket
\begin{align}
\brr{X+\al,Y+\be}=\brr{X,Y}
+ \Lie_X\be-\Lie_Y\al + \frac{1}{2}\d(\eval{\al,Y}-\eval{\be,X}),\label{CourantBracketDef}
\end{align}
$ X, Y\in \X(M), \al, \be\in \Omega^1(M)$, and the anchor $\ds \rho:\Tt M\to TM$
turn $\Tt M$ into a Courant algebroid. Now recall:

\begin{defn}
\label{def:Dirac:structure}
A subbundle $L$ of $\Tt M$ is called {isotropic} if it is isotropic under the symmetric form (\ref{eq:symmetric:bilinear:form}). It is called {integrable} if its space of sections $\Gamma(L)$ is closed under the Courant bracket (\ref{CourantBracketDef}). An integrable maximally isotropic subbundle of $\Tt M=TM\oplus T^*M$ is a {\bf Dirac structure} and, in this case, the pair $(M,L)$ is called a \textbf{Dirac manifold}.
\end{defn}

The maximal isotropy of a Dirac structure  $L$ is encoded in its characteristic equations:
\begin{equation}\label{eq:characteristic}
L\cap TM=\rho_*(L)^0, \quad \quad L\cap T^*M=\rho(L)^0.
\end{equation}
The integrability of $(M,L)$ provides a Lie algebroid structure on the bundle $L\to M$ with bracket
$\brr{\, , \,}|_L$ and anchor $\rho|_L$.

A function $f\in C^\infty(M)$ for which there exists a vector field  $X_f\in\X(M)$ such that
$X_f+d f\in \Gamma(L)$ is called  $L$-admissible.
The set of all $L$-admissible functions is denoted by $C_L^\infty(M)$ and has a Poisson bracket given by
$$
\{f,g\}=X_f(g)=-X_g(f).
$$

The fiber bundle $D_L=L\cap TM$ defines a (singular) integrable distribution over $M$, that we  call the characteristic
 distribution of $L$. We will denote by $\F_L$ (or simply $\F$, if there is no ambiguity) the foliation of $M$ corresponding to $D_L$.

 \subsection{Reducible Dirac manifolds}
 We will be interested in Dirac manifolds for which $D_L$ is a regular distribution (i.e. has constant rank). In particular,  the ones giving rise to a simple foliation $\F$ deserve a special name:

\begin{defn}
A Dirac manifold $(M,L)$ is said to be \textbf{reducible} if the characteristic distribution $D_L=L\cap TM$ is regular and induces a simple foliation $\F$. This means that the quotient  $M/\F$ is a smooth manifold and the projection $p:M\to M/\F$ is a submersion.
\end{defn}

For a reducible Dirac manifold $(M,L)$, with characteristic distribution  $D_L=L\cap TM$ and simple foliation $\F$, the characteristic equation  $D_L^0=\rho_*(L)$ and the exact sequence
$$0\rightarrow D_L \rightarrow TM \rightarrow p^!T(M/\F) \rightarrow 0$$
guarantees that  $\rho_*(L)$ is  isomorphic to the pullback vector bundle $p^!T^*(M/\F)$.

Moreover, the the injective map $p^*: C^\infty(M/\F)\to C_L(M)$ identifies the set of $L$-admissible functions with $C^\infty(M/ \F)$.
This defines a Poisson bracket on $C^\infty(M/ \F)$ so that $M/\F$ is a Poisson manifold with Lie algebroid $T^*(M/\F)$.
In fact, there is a one-one correspondence between reducible Dirac manifolds $(M,L)$ and Poisson structures on the quotient manifold $M/\F$ (see \cite{Courant, LWX}).

Given any Dirac manifold  $(M,L)$ with a regular characteristic distribution one can always choose a bivector $\Pi$ on $M$
such that (see \cite{Liu}):
$$
L=\set{X+\Pi^\sharp\al  + \al:\, X\in D_L, \al \in D_L^0}=D_L\oplus \mbox{graph}(\Pi^\sharp|_{D_L^0}).
$$
The pair $(D_L, \Pi)$ is called a {\bf characteristic pair} of $L$. In general, this pair is not unique. When $L$ is reducible, given any characteristic pair $(D_L, \Pi)$, one has:
$$
p_* \Pi=\pi.
$$
where $p:M\to M/\F$ is the projection and $\pi$ is the bivector determined by the Poisson bracket on $C^\infty(M/\F)$.

\begin{prop}\label{prop:pullback:reduced}
Let $(M,L)$ be a reducible Dirac manifold with characteristic foliation $\F$ and projection $p:M\to M/\F$. Then
 $L$ is isomorphic to the pullback Lie algebroid $p^{!!}T^*(M/\F)$.
\end{prop}

For the proof of we need the following Lemma which follows from a straightforward computation using Cartan's magic formula.

\begin{lem}\label{lem:pullback}
Let $p:M\to N$ be a smooth map and $\pi\in\X^2(N)$ a bivector on $N$. Then
$$
fg\, p^*\brr{\al,\be}_\pi = g\Lie_X(p^*\be) - f \Lie_Y(p^*\al) + \d\eval{fp^*\al,Y},
$$
for
$X, Y\in \X(M)$, $f,g\in C^\infty(M)$ and $\al,\be\in\Omega^1(N)$,  such that $\ds p_*X_x=f(x) \,\pi^\sharp(\al)_{p(x)}$ and $\ds p_*Y_x= g(x)\,\pi^\sharp(\be)_{p(x)}$, $x\in M$.
\end{lem}

\begin{proof}[Proof of Proposition \ref{prop:pullback:reduced}]
Let $(D,\Pi)$ be a characteristic pair of $L$ and let $\pi$ the Poisson bivector on $M/\F$ induced by $L$. Then
 $$L=\set{X  + \Pi(\al) + \al:\, X\in D, \al \in D^0} \mbox{ and } \, p_*\Pi=\pi.$$
By definition, the pullback Lie algebroid of $T^*(M/\F)$ is the bundle
$$p^{!!}T^*(M/\F)=\set{(X,\al) \in T_xM\oplus T^*_{p(x)}(M/\F): p_* X=\pi^\sharp\al, x\in M},$$
and the bundle map $\Psi:\ds p^{!!}T^*(M/\F)\to L$,
$$
X+\al \longmapsto (X,p^*\al)
$$
gives an isomorphism between $\ds p^{!!}T^*(M/\F)$ and $L$ as vector bundles.

The anchors of $\ds p^{!!}T^*(M/\F)$ and $L$ are clearly related by $\Psi$. So consider $X,Y\in \X(M)$,  $\al, \be\in \Omega(M/\F)$ and $f,g\in C^\infty(M)$,  such that,
$\ds p_*X=f\pi^\sharp{\al}$ and
$\ds p_*Y=g\pi^\sharp{\be}$, we have
\begin{align*}
\Psi\brr{X+ f\otimes \al,Y+g\otimes \be}&=\Psi\left(\brr{X,Y} + X(g)\otimes \be - Y(f)\otimes\, \al + f g \brr{\al,\be}_\pi\right)\\
&=(\brr{X,Y},   X(g)\, p^*\be - Y(f)\, p^*\al + f g\, p^*\brr{\al,\be}_\pi).
\end{align*}
Using Lemma \ref{lem:pullback} we conclude that $$ \Psi\brr{X+ f\otimes \al,Y+g\otimes \be}=\brr{X+p^*\al, Y+p^*\be}_L$$
and, consequently, $\Psi$ is a Lie algebroid isomorphism.
\end{proof}

\subsection{Modular classes of Dirac manifolds} One defines the {\bf modular class of a Dirac manifold} $(M,L)$ to be $\mod L$, the modular class of the underlying Lie algebroid. If its characteristic distribution $D_L$ has constant rank, it is integrable, hence defines a Lie algebroid, and we have the modular class $\mod D_L$.

Recalling that for a distribution $T\F$ the modular class is the obstruction to the existence of a transverse volume form to $\F$, we immediately obtain:

\begin{prop}
Let $(M,L)$ be a reducible Dirac manifold with characteristic distribution $D_L$. Then $\mod D_L=0$.
\end{prop}
\begin{proof}
Any volume form on $M/\mathcal{F}$ pulls back to  a non-vanishing transverse volume form to $\F$. This means that $\mod D_L=0$.
\end{proof}

On the other had, Propositions \ref{prop:KLW:pullback} and \ref{prop:pullback:reduced} yield immediately:

\begin{prop}\label{prop:modular:class:reducible}
Let $(M,L)$ be a reducible Dirac manifold with characteristic foliation $\F$ and projection $p:M\to M/\F$. The modular class of $L$ is the pullback of the modular class of $T^*(M/\F)$.
\end{prop}

\begin{proof}
It is enough to observe that the (pullback) Lie algebroid morphism $p:L\to T^*(M/\F)$ vanishes. This is obvious since this morphism is the composition of the isomorphism $L\simeq p^{!!}T^*(M/\F)$ of Proposition \ref{prop:KLW:pullback}, with the Lie algebroid morphism $p^{!!}T^*(M/\F)\to T^*(M/\F)$ of Proposition \ref{prop:pullback:reduced}, whose modular class vanishes.
\end{proof}

\begin{ex}
The graph $L_\pi$ of a Poisson bivector $\pi$ on a manifold $M$ is a Dirac structure which is  isomorphic (as a Lie algebroid) to $T^*M$.  In this case, $$\mod (L_\pi)\simeq\mod T^*M= 2 \mod(M,\pi).$$
This is precisely the case where $D|_L=0$.
\end{ex}

\begin{ex}
The graph $L_\omega$ of a presymplectic 2-form $\omega$ on a manifold $M$ is a Dirac structure isomorphic to $TM$. In the reducible case, $M/\F$ is a symplectic manifold and $\mod (L_\omega)=0 $.
\end{ex}

Proposition \ref{prop:modular:class:reducible} can be extended to the more general setting of a Dirac structure $L$ of a {\bf Poisson manifold bialgebroid $(T^*M,TM,\pi)$}. We recall that a Dirac structure $L$ on the Poisson manifold Lie bialgebroid $(TM,T^*M,\pi)$ is an integrable maximally isotropic subbundle of $\Tt M=TM\oplus T^*M$ equipped with the bilinear symmetric bracket (\ref{eq:symmetric:bilinear:form}), the anchor $\ds a=\rho + \pi^\sharp\smalcirc\rho_*$ and the Courant  bracket
\begin{align*}\label{def:Courant:Bracket:Poisson:manifold}
\brr{X+\al,Y+\be}=&\brr{X,Y}+\Lie_\al Y-\Lie_\be X - \frac{1}{2}\d_\pi(\eval{\al,Y}-\eval{\be,X})\nonumber \\
&+ \set{\al,\be}_\pi
+ \Lie_X\be-\Lie_Y\al + \frac{1}{2}\d(\eval{\al,Y}-\eval{\be,X}),
\end{align*}
for $X+\al, Y+\be\in \Gamma(\Tt M)$.

Ordinary Dirac manifolds are simply Dirac structures on the Poisson manifold Lie bialgebroid $(TM, T^*M, 0)$.

\begin{prop} \label{prop:isomorphism:Dirac:pullback}
Let  $(M,\pi)$ be a Poisson manifold and $L$ a Dirac structure on the Lie bialgebroid $(TM,T^*M,\pi)$. Consider the vector bundle map
 $\Phi:L\to \Tt M$  given by
\begin{equation*}\label{eq:map:isomorphism:Dirac:pullback}
X+\al\longmapsto \Phi(X+\al)=X+\pi^\sharp \al + \al.
\end{equation*}
The image of $L$ by $\Phi$ defines an ordinary Dirac structure on $M$ with the same characteristic foliation as $L$. Moreover $\Phi(L)$ and $L$ are isomorphic as Lie algebroids.

In particular, if $L$ is reducible then it is isomorphic to $p^{!!}T^*(M/\F)$, where $\F$ is the characteristic foliation of $L$ and $p:M\to M/\F$ is the natural projection.
\end{prop}

\begin{proof}
The map $\Phi$ is clearly a bijection, so it remains to prove that it preserves  Lie algebroid structures.
Let $X+\al,Y+\be\in\Gamma(L)$, then
\begin{align*}
\Phi\brr{X+\al,Y+\be}_L&=\Phi(\brr{X,Y}+ \Lie^*_\al Y - i_\be \d_\pi X + \brr{\al,\be}_\pi + \Lie_X\be-i_Y\d \al)\\
&=\brr{X,Y}+ \Lie^*_\al Y - i_\be \d_\pi X + \pi^\sharp(\brr{\al,\be}_\pi
+ \Lie_X\be-i_Y\d \al)\\
&+ \brr{\al,\be}_\pi + \Lie_X\be-i_Y\d \al.
\end{align*}
Since $\pi$ is a Poisson tensor, we have $\pi^\sharp\brr{\al,\be}_\pi=\brr{\pi^\sharp\al,\pi^\sharp\be}$ and
\begin{align*}
\Phi\brr{X+\al,Y+\be}&=\brr{X+\pi^\sharp \al,Y+\pi^\sharp\be} + \Lie_{X+\pi^\sharp\al}\be-i_{Y+\pi^\sharp\be} \d \al\\
&=\brr{\Phi(X+\al),\Phi(Y+\be)}.
\end{align*}
Finally, notice that anchors are preserved:
$$
a_L(X+\al)=X+\pi^\sharp\al=a_{\Phi(L)}(X+\pi^\sharp\al + \al)=a_{\Phi(L)}\smalcirc\Phi(X+\al),\quad X+\al\in L.
$$
\end{proof}

Since modular classes of Lie algebroid isomorphisms vanish, we conclude again that:

\begin{cor}
The modular class of any reducible Dirac structure $L$ on a Poisson manifold Lie bialgebroid $(TM, T^*M,\pi)$ is obtained by  pulling back  the modular class of $T^*(M/\F_L)$.
\end{cor}


\section{Modular class of a Dirac map}

We now turn our attention to the modular class of Dirac maps.

\subsection{Dirac maps}
We begin by recalling some definitions and properties of Dirac maps.

\begin{defn} Let $(M,L)$ and $(N,K)$ be two Dirac manifolds and $\phi:M\to N$ a smooth map.
\begin{enumerate}
\item[(i)] The \textbf{$b$-image}, or the backward image, of $K$ by $\phi$ is the bundle over $M$:
$$
\B_\phi(K)=\set{X+\phi^*\al\in T_x M\oplus T^*_x M:\, \phi_* X + \al \in K_{\phi(x)}, \, x\in M }.
$$
We say that $\phi:(M,L)\to (N,K)$ is a \textbf{$b$-Dirac map}, or a backward Dirac map, if
\begin{equation*}
L=\B_\phi(K).
\end{equation*}
\item[(ii)] The \textbf{$f$-image}, or the forward image, of $L$ by $\phi$ is the bundle over $M$:
$$
\F_\phi(L)=\set{\phi_*X+\al\in T_{\phi(x)}N\oplus T^*_{\phi(x)} N:\,  X + \phi^*\al \in L_{x}, \, x\in M }.
$$
We say that $\phi:(M,L)\to (N,K)$ is an \textbf{$f$-Dirac map}, or a forward Dirac map, if
\begin{equation*}
\phi^!K=\F_\phi(L).
\end{equation*}
\end{enumerate}
\end{defn}

In general, a forward Dirac map is not a backward Dirac map, nor vice-versa. However, for diffeomorphisms these two notions are equivalent. A diffeomorphism which is also a Dirac map is called a \textbf{Dirac diffeomorphism}.

\begin{ex}
Let $(M,L)$ be a Dirac manifold and $i:S\hookrightarrow M$ a submanifold.
Then $L|_S:=\B_i(L)$ is always a Dirac structure on $S$ and $i:(S, L|_S)\to (M,L)$ is a $b$-Dirac map. For example, if $\omega$ is a presymplectic form on $M$ and $L_\omega$ is its graph, then $L|_S$ is the graph of the pullback $i^*\omega$.
\end{ex}

\begin{ex}
For a reducible Dirac manifold $(M,L)$, the projection $p:M\to M/\F$ is both an $f$-Dirac and a $b$-Dirac map between $(M,L)$ and $(M/\F, \graf \pi_{M/\F})$.
\end{ex}

\begin{ex}(Poisson Dirac submanifold)
Let $(M,\pi_M)$ be a Poisson manifold and $i:S\hookrightarrow M$ a Poisson Dirac submanifold (see  \cite{CrainicLF}). This means that $S$ has a Poisson structure $\pi_S$ such that the symplectic foliation of $S$ is given by the intersection of the symplectic foliation of $M$ with $S$.
Then
  $i:(S,\graf \pi_S)\to (M,\graf \pi_M)$ is a $b$-Dirac map (and $TS\cap \pi_M(TS^0)=\set{0}$).
\end{ex}


The next two propositions show how $b$-Dirac maps and $f$-Dirac maps between reducible Dirac manifolds can be pushed to the quotient Poisson manifold.

\begin{prop}
A $b$-Dirac map between reducible Dirac manifolds that preserves the characteristic foliation induces a $b$-Dirac map between the graphs of the Poisson structures on the reduced manifolds.
\end{prop}

\begin{proof}
Let $\phi:(M,L)\to (N,K)$ be a $b$-Dirac map  preserving the characteristic foliations $\F_L$ and $\F_K$, i.e., $\phi(\F_L)\subset\F_K$. Then we have an induced map $\bar\phi:M/\F\to N/\F_K$ and we need to prove that:
\[ \B_{\bar\phi}(L_{\pi_{N/\F_K}})=L_{\pi_{M/{\F_L}}}. \]

For this, we choose a characteristic pair $(D|_L,\Pi)$ for $(M,L)$, let $x\in M$ and choose any element
\[ \pi_{M/\F_L}^\sharp(\al) + \al\in L_{\pi_{M/\F_L}}|_{p(x)}. \]
Then $p^*\al\in D_L^0$ and $\Pi^\sharp(p^*\al) + p^*\al\in L_x$. Since $\phi$ is a $b$-Dirac map, there exists $\be\in D_K^0$ such that $\phi^*\be=p^*\al$ and $\phi_*\Pi^\sharp(\phi^*\be) + \be\in K_{\phi(x)}$. If $q:N\to N/\F_K$ is the projection, then there is  $\tilde\be\in T_{q(\phi(x))}^*(N/\F_K)$ such that $q^*\tilde\be=\be$ and
$
q_*\phi_*\Pi^\sharp(\phi^* q^*\tilde\be)=\pi_{N/\F_K}\tilde\be.
$
Since $\bar\phi\smalcirc p=q\smalcirc\phi$, we have  $\bar\phi_*\pi_{M/\F}\bar\phi^*\tilde\be=\pi_{N/\F_K}\tilde\be$ and $p^*\al=\phi^*q^*\tilde\beta=p^*\bar{\phi}^*\tilde\beta$, which gives $\al=\bar{\phi}^*\tilde\beta$ (since $p$ is a submersion).

In conclusion, we showed that for any element
\[ \pi_{M/\F_L}^\sharp(\al) + \al\in L_{\pi_{M/\F_L}}\]
we can find  $\tilde\be\in T^*(N/\F_K)$ with $\al=\bar{\phi}^*\tilde\beta$ and
\[ \bar\phi_*\pi^\sharp_{M/\F_L}(\al) + \tilde\be\in L_{\pi_{N/\F_K}}.\]
proving our claim.
\end{proof}

An $f$-Dirac map always preserves the characteristic foliation, so in this case we have:

\begin{prop}
An  $f$-Dirac map between reducible Dirac manifolds induces a Poisson map between the reduced Poisson manifolds and, consequently, an f-Dirac map between the graphs of the Poisson structures on the reduced manifolds.
\end{prop}

\begin{proof}
Let $\phi:(M,L)\to (N,K)$ be an $f$-Dirac map between reducible Dirac manifolds, i.e., for each $x\in M$,
$$
K_{\phi(x)}=\set{\phi_*X+\be\in T_{\phi(x)}N \oplus T^*_{\phi(x)}N: \, X+\phi^*\be\in L_x}=\F_\phi(L_x).
$$
It is clear that if $X\in D_L$ then $\d\phi(X)\in D_K$, so $\phi$ induces a map $\bar\phi:M/\F\to N/\F_K$.

Let $f,g\in C^\infty_K(N)\cong C^\infty(N/{\F_K})$, then $f\smalcirc \phi$ and $g\smalcirc\phi$ are $L$-admissible functions. Moreover
\begin{eqnarray*}
\set{f\smalcirc\phi, g\smalcirc	\phi}_L(x)&=&\eval{X_{f\smalcirc\phi}(x),d_x (g\smalcirc\phi)}=
\eval{\phi_* X_{f\smalcirc\phi}(x), d_{\phi(x)} g}\\
&=&\eval{X_f,d g} (\phi(x))=\set{f,g}_K(\phi(x)).
\end{eqnarray*}

This shows that $\phi$ preserves the Poisson bracket between admissible functions and, consequently, the map $\bar\phi:{M/{\F_L}}\to {N/{\F_K}}$ induced by $\phi$ is a Poisson map.
\end{proof}

Since Poisson maps induce Lie algebroid comorphisms between cotangent bundles \cite{Mackenzie} we conclude that:

\begin{cor} An f-Dirac map between reducible Dirac structures induces a Lie algebroid comorphism between the cotangent bundles of the reduced manifolds.
\end{cor}

\begin{prop}
If $\phi:(M,L)\to (N,K)$ is both an $f$-Dirac map and an immersion, then $\phi$ is also a $b$-Dirac map and defines a Lie algebroid comorphism between $L$ and $K$.
\end{prop}

\begin{proof}
Since $\phi$ is an $f$-Dirac map and an immersion, we may define the bundle map
$\ds \Psi:\phi^!K\to L$ given by:
$$
Y+\be \longmapsto  (\phi_*)^{-1} Y+\phi^*\be.
$$
To prove that $\Psi$ preserves brackets and anchors, for $X+\al,Y+\be\in\Gamma(K)$ we set:
\[ (X+\al)^!:=(X+\al)\smalcirc\phi\in\Gamma(\phi^!K), \quad (Y+\be)^!:=(Y+\be)\smalcirc\phi\in\Gamma(\phi^!K).\]
Then
\[ \phi_*\smalcirc a_L\smalcirc\Psi(X+\al)^!=\phi_*(\phi_*)^{-1}X = X=a_K(X+\al), \]
so $\Psi$ preserves anchors. Also, we find that:
\begin{align*}
\Psi\brr{(X+\al), (Y+\be)}^!&=\Psi\left(\brr{X,Y} + \Lie_X\be-\Lie_Y\al + d\eval{\al,Y}\right)^!\\
&= (\phi_*)^{-1}\brr{X,Y}^! + \phi^*\left(\Lie_X\be-\Lie_Y\al + \d\eval{\al,Y}\right)^!\\
&= \brr{(\phi_*)^{-1}X, (\phi_*)^{-1}Y} +  \Lie_{(\phi_*)^{-1} X}\phi^*\be+\\
&\qquad \qquad \qquad -\Lie_{(\phi_*)^{-1}Y}\phi^*\al + \d\eval{\phi^*\al,(\phi_*)^{-1}Y}\\
&= \brr{\Psi(X+\al)^! , \Psi(Y+\be)^!}
\end{align*}
so $\Psi$ also preserves brackets.
\end{proof}

The next proposition  shows that the pullback Lie algebroid of a Dirac manifold by a submersion is, in fact, a backward Dirac image.

\begin{prop}
Let $\phi:M\to N$ be a submersion and $(N,K)$ a Dirac manifold. Then $\mathcal{B}_\phi(K)$ is a Lie algebroid isomorphic to the pullback Lie algebroid $\phi^{!!}K$.
\end{prop}

\begin{proof}
Since $\phi$ is a submersion, the vector bundle map 
$\displaystyle \Psi:\mathcal{B}_\phi(K)\to \phi^{!!}K$ given by:
$$
X+\phi^*\be\longmapsto(X,\phi_*X + \be),
$$
is a vector bundle  isomorphism. It remains to prove that it is a Lie algebroid morphism, if $\mathcal{B}_\phi(K)$ is equipped with the bracket and the anchor induced by $\Tt M$.

It is easy to see that $\Psi$ preserves anchors:
$$a^!\smalcirc\Psi(X+\al)= X= a_{\mathcal{B}_\phi(K)}(X+\al), \quad X+\al\in\mathcal{B}_\phi(K).$$

To prove that $\Psi$ also preserves Lie brackets let $X+\al, Y+\be\in\Gamma(\mathcal{B}_\phi(K))$ such that $\Psi(X+\al)=\sum_i f_i\otimes(x_i+\mu_i)$ and $\Psi(Y+\be)=\sum_i g_i\otimes(y_j+\gamma_j)$, for some $f_i,g_j\in C^\infty(M)$, $x_i+\mu_i, y_j+\gamma_j\in\Gamma(K)$. Then
$ \phi_*X=\sum_if_ix_i$, $\al=\sum_i f_i\phi^*\mu_i$, $ \phi_*Y=\sum_j g_j y_j$ and $\be=\sum_j g_j\phi^*\gamma_j$.

Therefore
\begin{align*}
\Psi(\brr{X+\al,Y+\be}_{\mathcal{B}_\phi(K)})&= \Psi\left(\brr{X,Y} + \Lie_X \left(\sum_jg_j\gamma_j\right)+\right.\\
&\qquad\qquad\qquad\qquad \left.- \Lie_Y\left(\sum_i f_i\mu_i\right) +d\eval{\sum_if_i\gamma_i,Y}\right)\\
=\Psi(\brr{X,Y} +&\sum_{i,j} f_i g_j \phi^*\left(\Lie_{x_i}\gamma_j - \Lie_{y_j}\mu_i + d\eval{\mu_i,y_j}\right)+\\
& \qquad \qquad  + \sum_j X(g_j)\phi^*\gamma_j - \sum_i Y(f_i)\phi^*\mu_i)\\
=(\brr{X,Y}, \phi_*&\brr{X,Y} + \sum_{i,j} f_i g_j \left(\Lie_{x_i}\gamma_j - \Lie_{y_j}\mu_i + d\eval{\mu_i,y_j}\right) + \\
& \qquad \quad \quad +\sum_j X(g_j)\gamma_j - \sum_i Y(f_i)\mu_i)\\
=[\Psi(X+\al),&\Psi(Y+\be)]_{\phi^{!!}K}.
\end{align*}
\end{proof}

\begin{cor}\label{prop:b:Dirac:map:pullback}
If a submersion $\phi:(M,L)\to (N,K)$ is a $b$-Dirac map then it is also an $f$-Dirac map and the Lie algebroid
 $L$ is isomorphic to the pullback $\phi^{!!}K$.
\end{cor}

We also recover Proposition \ref{prop:pullback:reduced} in this way:

\begin{cor}
Let $(M,L)$ be a reducible Dirac manifold and $p:M\to M/\F$ the projection, then $L$ is a Lie algebroid isomorphic with $p^{!!}L_{\pi_{M/\F}}\cong p^{!!}T^*(M/\F)$.
\end{cor}

\begin{cor}
Let $\phi: (M,L)\to (N,K)$ be both a submersion and $b$-Dirac map between reducible Dirac manifolds. Then it induces a Lie algebroid morphism and a Lie algebroid comorphism  between $T^*(M/{\F_L})$ and $T^*(N/{\F_K})$.
\end{cor}

\subsection{Modular classes of Dirac maps}
We are able to compare the modular classes of Dirac manifolds related by Dirac maps, for the following special class:

\begin{defn}
A Dirac map $\phi:(M,L)\to (N,K)$ is called \textbf{admissible} if
\begin{equation*}
\R_{(x,\phi(x)) }=\set{\left(X+\phi^*\al,\phi_* X+\al\right): X+\phi^*\al\in L_x, \, \phi_*X+\al\in K_{\phi(x)}} \quad (x\in M),
\end{equation*}
defines a vector subbundle of $\Tt M\times \Tt N$ over $\graf \phi$.
\end{defn}

\begin{ex}
An  $f$-Dirac map $\phi:(M,L)\to (N,K)$ such that $L \cap \ker\phi_*$ has constant rank  is admissible. In particular, if $L$ is the graph of a Poisson bivector, then $\phi$ is admissible.
\end{ex}

\begin{ex}
A  $b$-Dirac map $\phi:(M,L)\to (N,K)$  such that $\ker\phi^* \cap \phi^! K$  has constant rank  is admissible. In particular, if $K$ is the graph of a pre-symplectic  form, then $\phi$ is admissible.
\end{ex}

For any smooth map $\phi:M\to N$, the vector subbundle of $\Tt M\times \Tt N$
$$
\Gamma^\phi=\set{(v+\phi^*\al, \phi_*v+\al)\in \Tt_x  M \times \Tt_{\phi(x)}N: \, x\in M}\subset \Tt(M\times N)
$$
is a Dirac structure supported on $\graf \phi$ \cite{AXu,BPS}. It has a natural Lie algebroid structure over  $\graf \phi$, where the anchor is the projection
$$
a\left(X+\phi^*\al,\phi_* X+\al\right)=(X,\phi_* X),
$$
and the  Lie bracket on sections is induced by the Courant bracket on $\Tt (M\times N)$.

\

Notice that for an admissible map $\phi$, the Lie algebroid structure on $\Gamma^\phi$ induces  a natural Lie  algebroid structure on $\ds \R$  (over $\graf \phi$) and we have:
\begin{prop}
Let $\phi:(M,L)\to (N,K)$ be an admissible Dirac map.
The vector bundle $\R$ has a Lie algebroid structure over $\graf \phi$ and
the projection maps
\begin{equation}
\label{diag:mod:class}
\xymatrix{
\R\ar[r]^i \ar[dr]_{j}& L \\
&K}
\end{equation}
are Lie algebroid morphisms.
\end{prop}

The Lie algebroid morphisms $i$ and $j$ induce morphisms at the cohomology level $i^*:H^*(L)\to H^*(\R)$ and $j^*:H^*(L)\to H^*(\R)$. This leads to a natural definition of modular class of Dirac maps, which is a class comparing the modular classes of $L$ and $K$ under a Dirac map.

\begin{defn}
The \textbf{modular class of an admissible Dirac map} $\phi:(M,L)\to (N,K)$ is the class:
\begin{equation*}
\mod \phi= i^*\mod L-j^*\mod K\in H^1(\R).
\end{equation*}
\end{defn}

\begin{ex}
Let $\phi:(M,\pi_M)\to (N,\pi_N)$ be a Poisson map. Then $\phi$ defines an $f$-Dirac map between $(M,L_{\pi_M})$ and $(N,L_{\pi_N})$.
In this case, the Lie algebroid $\R$ is naturally isomorphic to the Lie algebroid $\phi^!(T^*N)$ and  the modular class of the Dirac map $\phi$ coincides with  the modular class of the comorphism defined by $\phi$ (under this isomorphism).
\end{ex}

\begin{ex}Let $(M,\pi_M)$ be a Poisson manifold and $\phi:S\hookrightarrow M$ a Poisson Dirac submanifold.
When  $\displaystyle \dim(T_xS^0\cap \ker\pi_M|_x)$ does not depend on $x\in S$, one says that $S$ is a Poisson Dirac submanifold of constant rank. In this case the inclusion $\phi:(S,L_{\pi_S})\to (M,L_{\pi_M})$ is an admissible b-Dirac map.

In order to see the meaning of the modular class of $\phi$ observe that it lies in the first cohomology group of the Lie algebroid (over $\graf \phi$):
  $$
  \R_{(x,\phi(x)) }=\set{\left(\pi_S^\sharp \phi^* \al+\phi^*\al, \pi_M^\sharp \al+\al\right): \al\in \ker_x(\phi_*\pi_S^\sharp \phi^*-\pi_M)}, \quad x\in S
  $$
which is clearly isomorphic to the Lie algebroid (over $S$)
$$ \mathcal{K}=\ker (\phi_*\pi_S \phi^*-\pi_M)\subset \phi^!(T^*M)=T^*_SM.$$
Notice that this Lie algebroid  $\mathcal K$ is equipped with the Lie bracket
  $$
  \brr{\al^!, \be^!}=\brr{\al,\be}_M^!
  $$
  and the anchor $\ds a(\al^!)=\pi_M^\sharp(\al)$, where $\al^!=\al\smalcirc \phi\in \Gamma(\phi^!(T^*M))$ and $\be^!=\be\smalcirc \phi\in\Gamma(\phi^!(T^*M))$, for  $\al, \be\in \Omega^1(M)$. From diagram (\ref{diag:mod:class}) we have
  \begin{equation*}
\xymatrix{
\R\cong \mathcal{K}\ar[r]^{i} \ar[dr]_{j}& T^*S \\
&T^*M}
\end{equation*}
so we see that the modular class of $\phi$ is given by
$$ i^* \mod T^*S - j^* \mod T^*M=2(i^* \mod (S,\pi_S) - j^* \mod (M,\pi_M))\in H^1(\mathcal K).$$

The isomorphism above shows that when $\phi:S\hookrightarrow M$ is a Poisson submanifold, then $\mod\phi$ coincides with the relative modular class already introduced in \cite{CLF}.
\end{ex}

It follows from the definition that a Dirac diffeomorphism has vanishing modular class, since in that case $i$ and $j$ in (\ref{diag:mod:class}) are both Lie algebroid isomorphism. This is also clear from the following proposition:

\begin{prop}\label{prop:b:f:Dirac:map:vanish:modular:class}
Let $\phi:(M,L)\to (N,K)$ be a $b$-Dirac map, which is also a submersion. Then $\mod \phi=0$.
\end{prop}

\begin{proof}
Since $\phi$ is a submersion and a $b$-Dirac map, then by Corollary \ref{prop:b:Dirac:map:pullback} it is also an $f$-Dirac map and  $\phi^{!!}K\cong L$ (as Lie algebroids). By Proposition \ref{prop:KLW:pullback}, we know that the modular class of the pullback map
$$
\Psi:\phi^{!!}K\cong L\to K
$$
vanishes. This means  $\mod L=\Psi^*\mod K$ and we obtain:
$$\mod \phi=i^*\mod L - j^*\mod K=(\Psi\smalcirc i)^*\mod K-j^*\mod K=0.$$
\end{proof}

\begin{cor}
Let $(M,L)$ be a reducible Dirac manifold with characteristic foliation $\F$. The projection $p:(M,L)\to (M/\F,L_\pi)$ has vanishing modular class.
\end{cor}

\begin{ex}
Let $(M,\pi)$ be a Poisson manifold,  $(G,\pi_G)$ a Poisson-Lie group and
let $G\times M\to M$ be a Poisson action which is hamiltonian in the sense of Lu \cite{Lu1}. This means
that there exists a smooth map $\kappa:M\to G^*$ with values in the dual, 1-connected, Poisson-Lie group, which is equivariant
relative to the left dressing action of $G$ on $G^*$, and satisfies the moment map condition:
\begin{equation*}
\xi_M=\pi^\sharp(\kappa^*\xi^L),\quad \forall\ \xi\in\gg
\end{equation*}
where $\xi^L\in\Omega^1(G^*)$ is the left-invariant 1-form satisfying $\xi^L|_e=\xi\in T^*_eG^*=\gg$. 

When the action is proper and free, one has a unique quotient Poisson structure on $M/G$ for which the projection $\phi:M\to M/G$ is a surjective, Poisson submersion whose modular class vanishes (see \cite{CLF}). In this case the identity $e\in G^*$ is a regular value of $\kappa$ and the Dirac manifold $(M,L_\pi)$  pulls back to a Dirac structure $L_e$ on the level set $\kappa^{-1}(e)$ (see \cite{Bursztyn}). The characteristic distribution of $L_e$  coincides with the distribution tangent to the $G$-orbits on $\kappa^{-1}(e)$. This means that
$$\ds M//G:=\kappa^{-1}(e)/G$$
inherits  a Poisson structure $\pi_{\red}$ induced by $L_e$ and $\bar\phi=\phi_{|_{\kappa^{-1}(e)}}:(\kappa^{-1}(e),L_e)\to (M//G, L_{\pi_{M//G}})$ is both a $b$-Dirac map and an $f$- Dirac map. Moreover, by Proposition \ref{prop:pullback:reduced}, one has $\ds L_e\cong\bar\phi^{!!}T^*(M//G)$ and $\mod \bar\phi=0$. Therefore, we obtain a commutative diagram:
\[
\xymatrix{
     &M\ar[dr]\\
\kappa^{-1}(e)\ar[ur]\ar[dr]& & M/G\\
 & M//G\ar[ur]}
\]
where the inclusions are $b$-Dirac maps and the projections are both $b$-Dirac maps and $f$-Dirac maps, between the induced Dirac manifolds.

Now Proposition \ref{prop:b:f:Dirac:map:vanish:modular:class}  shows that the modular classes of the forward Dirac maps in the above diagram vanish. This gives a more direct and simple explanation for the result in \cite{CLF}, where we have already proved, without making use of Dirac structures, that the projection $\phi:M\to M/G$ has vanishing modular class (as a Poisson map) and that there exists a modular vector field on $M$, tangent to $\kappa^{-1}(e)$, whose projection  onto $M//G$ is a modular vector field of this Poisson manifold.
\end{ex}


\end{document}